\documentclass[11pt]{amsart}

\usepackage[utf8]{inputenc}
\usepackage{microtype}
\usepackage{amsfonts}
\usepackage{amsmath}
\usepackage{graphicx}
\usepackage{float}
\usepackage[dvipsnames]{xcolor}
\usepackage{hyperref}
\usepackage{tikz-cd}
\usepackage{comment}

\usepackage{amssymb}
\usepackage{enumitem}
\usepackage{mathrsfs}

\usepackage{tikz}
\usetikzlibrary{topaths}
\usetikzlibrary{calc}
\usetikzlibrary{positioning,shapes,shadows, decorations.markings}
\usetikzlibrary{arrows,automata}
\usepackage{xcolor}
\usepackage{tikz-cd} 
\tikzset{
  mid arrow/.style={
    postaction={decorate},
    decoration={
      markings,
      mark=at position 0.6 with {\arrow{to}}
    }
  }
}

\usepackage{a4wide}

\usepackage{empheq}

\newtheorem{theorem}{Theorem}[section]
\newtheorem*{theorem*}{Theorem}
\newtheorem{lemma}[theorem]{Lemma}
\newtheorem{proposition}[theorem]{Proposition}
\newtheorem*{proposition*}{Proposition}
\newtheorem{corollary}[theorem]{Corollary}
\newtheorem*{corollary*}{Corollary}

\newtheorem{cit}[theorem]{Citation}
\newtheorem*{conjecture*}{Conjecture}

\newtheorem*{question*}{Question}

\newtheorem{lettertheorem}{Theorem}

\newtheorem{lettercorollary}[lettertheorem]{Corollary}

\theoremstyle{definition}
\newtheorem{definition}[theorem]{Definition}
\newtheorem*{definition*}{Definition}
\newtheorem{remark}[theorem]{Remark}
\newtheorem*{remark*}{Remark}
\newtheorem{example}[theorem]{Example}

\newcommand{\N}{\mathbb{N}}
\newcommand{\Z}{\mathbb{Z}}
\newcommand{\Q}{\mathbb{Q}}
\newcommand{\R}{\mathbb{R}}

\newcommand{\AG}{\mathcal{AG}}
\newcommand{\Cantor}{\mathfrak{C}}
\newcommand{\Graph}{\mathcal{QT}}
\newcommand{\cycle}{\mathcal{C}}

\DeclareMathOperator{\Aut}{Aut}

\DeclareMathOperator{\Homeo}{Homeo}

\DeclareMathOperator{\F}{F}
\DeclareMathOperator{\Supp}{Supp}
\DeclareMathOperator{\Stab}{Stab}
\DeclareMathOperator{\Cay}{Cay}

\numberwithin{equation}{section}

\begin{document}

\title[Action graphs, semiconjugacy, and non-embedding in $V$]{Action graphs, semiconjugacy, and non-embedding in Thompson's group $V$}
\date{\today}


\author[J.~Hyde]{James Hyde}
\address{Department of Mathematics and Statistics, Binghamton University, Binghamton, NY}
\email{jhyde1@math.binghamton.edu}

\author[R.~Skipper]{Rachel Skipper}
\address{University of Utah, 155 S 1400 E, Salt Lake City, UT}
\email{rachel.skipper@utah.edu}

\author[M.~C.~B.~Zaremsky]{Matthew C.~B.~Zaremsky}
\address{Department of Mathematics and Statistics, University at Albany (SUNY), Albany, NY}
\email{mzaremsky@albany.edu}

\begin{abstract}
We prove a variety of results about subgroups of Thompson's group $V$. First we prove that every action graph of a finitely generated subgroup of $V$ acting on an orbit in Cantor space is quasi-isometric to a tree. Then we prove that for a broad class of groups of homeomorphisms of the real line, for example Thompson's group $F$, any action on the Cantor space via an embedding into Thompson's group $V$ must be semiconjugate to the standard action on the line. Finally, we use this to establish that many such groups cannot embed into $V$; in particular the Stein group $F_{2,3}$ cannot embed in $V$, answering a question of the third author.
\end{abstract}

\maketitle
\thispagestyle{empty}

\section*{Introduction}

Thompson's group $V$ is the group of homeomorphisms $h$ of Cantor space $\Cantor=\{0,1\}^\N$ that are locally given by prefix replacements, i.e., for all $\kappa\in \Cantor$ there exist a finite prefix $w$ of $\kappa$ and another finite string $u$ such that $(w\kappa')h=u\kappa'$ for all $\kappa'\in \Cantor$. (Throughout this paper we use right actions, and correspondingly write functions to the right of their arguments.) This group, together with its subgroup $T$, was the first known example of a finitely presented infinite simple group, and a significant amount of research has been dedicated to understanding it.

Our first main result of this paper is that every action graph of a finitely generated subgroup $G\le V$ acting on an orbit in $\Cantor$ is a quasi-tree. More precisely:

\begin{lettertheorem}\label{thrm:q_tree}
Let $G$ be a finitely generated subgroup of $V$, with $S$ a finite symmetric generating set for $G$. Let $X$ be the $G$-orbit of some $\kappa_0\in\Cantor$, and let $\Gamma$ be the action graph for $G$ acting on $X$, that is, the Schreier graph with respect to the subgroup $\Stab_G(\kappa_0)$. Then $\Gamma$ is quasi-isometric to a tree.
\end{lettertheorem}

When $G=V$, this result is easy to see, but it is somewhat surprising that it holds for all finitely generated $G\le V$. Notably, it holds despite the inclusion $G\to V$ sometimes failing to induce a quasi-isometric embedding of action graphs.

\medskip

Our other main results concern Thompson's group $F$ and some of its generalizations. This is the group of piecewise-linear orientation-preserving homeomorphisms of the interval $[0,1]$ that have breakpoints in $\Z[\frac{1}{2}]$ and slopes powers of $2$. Using the usual lexicographic order on $\Cantor$ coming from the order $0<1$ on $\{0,1\}$, it turns out that $F$ is isomorphic to the subgroup of all order preserving elements of $V$. A wealth of background information on $F$ can be found for example in \cite{cannon96,belk04}. These two viewpoints of $F$ are connected by a semiconjugacy from $\Cantor$ to $[0,1]$. Here a \emph{semiconjugacy} $X\to Y$ between two topological spaces equipped with actions of a group $G$ is a continuous surjective $G$-equivariant map. Namely, for $\kappa=\epsilon_1 \epsilon_2 \epsilon_3 \cdots \in \Cantor$ ($\epsilon_i\in\{0,1\}$), the map sending $\kappa$ to the real number $\frac{\epsilon_1}{2} + \frac{\epsilon_2}{4} + \frac{\epsilon_3}{8} + \cdots$ is a continuous surjection $\Cantor \to [0,1]$, and it turns out that this map is $F$-equivariant. It will be more convenient in all that follows to consider actions on $\R$ and then use the one-point compactification $S^1=\R\cup\{\infty\}$ of $\R\cong(0,1)$ so that there is a unique global fixed point. Then the above should be viewed as a semiconjugacy from $\Cantor$ to $S^1$. (Strictly speaking semiconjugacies are between two actions, but when the actions are understood we will just refer to a semiconjugacy between two spaces.)

\begin{remark*}
The term ``semiconjugacy'' is not entirely uniquely defined in the literature, so let us emphasize what we mean. For us a semiconjugacy is a continuous surjective $G$-equivariant map from one $G$-space to another. Sometimes in the literature semiconjugacies are only considered from a $G$-space to itself. Additionally, continuity is sometimes not included in the definition, with the term topological semiconjugacy used when the assumption of continuity is included. Since our results are of the form, ``these things are semiconjugate,'' using a stronger definition gives us stronger results.
\end{remark*}

The second main result of this paper is that, for $F$ and groups similar to $F$, not only is this ``standard'' copy of the group inside $V$ semiconjugate to the standard copy acting on $S^1=\R\cup\{\infty\}$, but \emph{every} copy of the group inside $V$ is semiconjugate to the standard copy on $S^1$.

\begin{lettertheorem}\label{thrm:main}
Let $E\le\Homeo_+(\R)$ be finitely generated, locally moving with abelian germs (Definitions~\ref{def:lm} and~\ref{def:ag}), and FG-filtered (Definition~\ref{def:FG_filtered}). Then for any abstract embedding of $E$ into $V$, the resulting action of $E$ on $\Cantor$ is semiconjugate to the standard action of $E$ on $S^1$.
\end{lettertheorem}

The quintessential example of a group satisfying these hypotheses is Thompson's group $F$, so we conclude that for every copy of $F$ in $V$, the action of this copy on $\Cantor$ is semiconjugate to the standard action of $F$. We remark that there do exist copies of $F$ in $V$ that are not conjugate to the standard one (see Example~\ref{ex:crazy_F}).

As we said, it is convenient to work with the circle $S^1=\R\cup\{\infty\}$, so that there is a unique global fixed point. Given the focus on $S^1$, one might wonder for example whether Theorem~\ref{thrm:main} holds for things like Thompson's group $T$. However, since semiconjugacies send global fixed points to global fixed points, this cannot hold if the given action of $T$ fixes a point in $\Cantor$, which is easy to achieve. It would be interesting to try and formulate an appropriate statement about copies of $T$ in $V$, but we will not pursue this here.

There is quite a bit of literature on situations where arbitrary actions on the line or circle are semiconjugate to standard actions. For example, in \cite{brum21} the authors prove many results involving semiconjugacies from the line to itself, typically for groups that are locally moving, meaning the rigid stabilizer of each open interval has no global fixed points in the open interval. For example \cite[Theorem~5.3.5]{brum21} says that every action of Thompson's group $F$ on $\R$ by $C^1$ diffeomorphisms without discrete orbits is semiconjugate to the standard action. Another example is \cite[Theorem~7.2.2]{brum21}, which provides semiconjugacy rigidity for certain groups acting on the circle. We should also mention fundamental work of Rubin in \cite{rubin89,rubin96} on strong rigidity properties of locally dense actions; see also \cite{belk25}. Finally, see \cite{hyde23,lodha20} for related work on the interval and circle including non-embedding results.

\medskip

Our main application of Theorem~\ref{thrm:main} is to rule out certain groups (abstractly) embedding in $V$. For example, we prove that the Stein group $F_{2,3}$ cannot embed in $V$, answering a question of the third author.

\begin{lettertheorem}\label{thrm:non_embed}
Let $E\le \Homeo_+(\R)$ be robust (Definition~\ref{def:robust}), FG-filtered (Definition~\ref{def:FG_filtered}), and tame at $\infty$ (Definition~\ref{def:tame}). If $E$ embeds in $V$, then every finitely generated subgroup of the group of germs of $E$ at $\infty$ is cyclic.
\end{lettertheorem}

Here the \emph{group of germs at $\infty$} of $E$ is the quotient of $E$ by the normal subgroup of all elements $f$ such that $f$ fixes some open ray $(a,\infty)$. We will also use the \emph{group of germs at $x$} of a group $G$ acting on a topological space $X$ with $x\in X$, which is the quotient of $\Stab_G(x)$ by the normal subgroup of all elements $g$ such that $g$ fixes an open neighborhood of $x$. In either case, the image of an element in the group of germs is the \emph{germ} of the element.

\begin{lettercorollary}\label{cor:stein_no_V}
The Stein group $F_{2,3}$ does not embed in $V$.
\end{lettercorollary}

The \emph{Stein group} $F_{2,3}$, first studied in depth by Stein in \cite{stein92} along with a broad family of related groups, is the group of all piecewise-linear orientation-preserving homeomorphisms of $[0,1]$ with breakpoints in $\Z[1/6]$ and slopes of the form $2^m 3^n$ ($m,n\in\Z$). Identifying $(0,1)$ with $\R$, the group of germs of $F_{2,3}$ at $\infty$ is isomorphic to $\Z^2$, i.e., not cyclic. The Stein group naturally contains $F$, and also the ternary relative $F_3$ of $F$. In \cite{lodha20}, Lodha proved that $F_{2,3}$ cannot abstractly embed as a subgroup of $F$ (see also \cite{hyde23} for another proof by the first author and Moore), so Corollary~\ref{cor:stein_no_V} is a strengthening of this. The proofs in \cite{lodha20} and \cite{hyde23} do not seem to have easy analogs working with Cantor space, so instead we use the semiconjugacies we construct in Theorem~\ref{thrm:main} to prove Theorem~\ref{thrm:non_embed} and Corollary~\ref{cor:stein_no_V}.

One source of interest in the question of which groups embed in $V$ is Lehnert's conjecture, which predicts that every co-context-free group embeds in $V$; see \cite{bleak16,lehnert_thesis,lehnert07}. We do not know whether $F_{2,3}$ is co-context-free, though we would expect it is not.

\subsection*{Acknowledgments} We are grateful to Corentin Bodart, Jos\'e Burillo, Mart\'in Gilabert Vio, Yash Lodha, Justin Moore, and Ran Tao for helpful conversations, and to Francesco Fournier-Facio for many helpful comments on a preliminary version. We also thank Jim Belk, Collin Bleak, and Martyn Quick for providing Example~\ref{ex:crazy_F}. RS is supported by NSF DMS-2506840.

\section{Action graphs}\label{sec:action}

Throughout, we view connected graphs as metric spaces by giving each edge length $1$ and using the usual path metric. We will mostly focus on vertex sets of graphs, with the induced metric coming from the path metric on the graph.

\begin{definition}[Action graph]
Let $G$ be a group acting on a set $X$, and let $S$ be a finite symmetric subset of $G$. The \emph{action graph} $\AG_X(G,S)$ associated to these data is the directed graph with vertex set $V(\AG_X(G,S))=X$ and a directed edge $(x,x.s)$ for each $x\in X$ and $s\in S$.
\end{definition}

For example if $X=G$ with the usual translation action, then $\AG_G(G,S)$ is the Cayley graph $\Cay(G,S)$. More generally if $X=H\backslash G$ for some $H\le G$ then the action graph is the Schreier graph. Note that an action graph is connected if and only if the subgroup of $G$ generated by $S$ acts transitively on $X$. In this case, the action graph is canonically isomorphic to the Schreier graph for the subgroup $H=\Stab_G(x)$ for any fixed $x\in X$.

Now we specialize to action graphs of subgroups of Thompson's group $V$.

\begin{definition}[The graph $\Graph_\kappa$]
Consider the directed graph with vertex set $\Cantor$ and directed edges $(\kappa,0\kappa)$ and $(\kappa,1\kappa)$ for each $\kappa\in\Cantor$ (allowing for loops). Note that the degree of each vertex is $3$, with one incoming and two outgoing edges. For a given $\kappa\in\Cantor$, write $\Graph_\kappa$ for the connected component of this graph that contains the vertex $\kappa$. Note that $\Graph_\kappa=\Graph_{\kappa'}$ if and only if $\kappa$ and $\kappa'$ differ by finite prefixes.
\end{definition}

The structure of the component $\Graph_\kappa$ depends on whether $\kappa$ is \emph{rational}, meaning of the form $\kappa = w\overline{u}$ for some finite words $w,u\in\{0,1\}^*$, or \emph{irrational}, meaning not of this form.

\begin{lemma}\label{lem:q_tree}
If $\kappa$ is irrational then $\Graph_\kappa$ is a $3$-regular tree. If $\kappa$ is rational then $\Graph_\kappa$ is the union of finitely many bounded degree rooted binary trees together with a cycle graph $\cycle_\kappa$ each of whose vertices is adjacent to one of these trees' roots. In particular $\Graph_\kappa$ is always a $3$-regular quasi-tree.
\end{lemma}

Here a \emph{quasi-tree} is a metric space quasi-isometric to a tree.

\begin{proof}
If $\kappa$ is irrational, then there is a unique reduced path between any two vertices, hence $\Graph_\kappa$ is a $3$-regular tree. Now suppose $\kappa$ is rational, say $\kappa=w\overline{u}$ for $u$ of minimum length and $u$ not a suffix of $w$. Since $\Graph_{w\overline{u}}=\Graph_{\overline{u}}$, without loss of generality $w$ is empty. Let $\cycle_\kappa$ be the finite set of vertices of the form $v\overline{u}$ for $v$ a suffix of $u$, so $\cycle_\kappa$ has $|u|$ many vertices. We cannot have both $0v$ and $1v$ be suffixes of $u$, so the induced subgraph on the vertex set $\cycle_\kappa$ is a cycle of length $|u|$. Each vertex $v\overline{u}$ of $\cycle_\kappa$ is adjacent to exactly one vertex outside $\cycle_\kappa$, whichever of $0v\overline{u}$ or $1v\overline{u}$ is not in $\cycle_\kappa$. Now the induced subgraph of $\Graph_\kappa$ with vertices obtained by adding prefixes to whichever of $0v\overline{u}$ or $1v\overline{u}$ is not in $\cycle_\kappa$ is a rooted binary tree. For different $v$, these trees are disjoint. Thus we obtain the desired description of $\Graph_\kappa$, and collapsing $\cycle_\kappa$ to a vertex defines a quasi-isometry to a tree.
\end{proof}

See Figures~\ref{fig:irrational_graph} and~\ref{fig:rational_graph} for examples in the irrational and rational cases.

\begin{figure}[htb]
\begin{center}
\begin{tikzpicture}[line width=1]
\draw[mid arrow] (3.5,-3.5) -- (1.5,-1.5);
\draw[mid arrow] (1.5,-1.5) --  (0.5,-0.5);
\draw[mid arrow] (0.5,-0.5) -- (0,0);
\draw[mid arrow] (0.5,-0.5) -- (1,0);
\draw[mid arrow] (1.5,-1.5) --  (2.5,-0.5);
\draw[mid arrow] (2.5,-0.5) --  (2,0);
\draw[mid arrow] (2.5,-0.5) --  (3,0);
\draw[mid arrow] (3.5,-3.5) -- (5.5,-1.5);
\draw[mid arrow] (5.5,-1.5) --  (6.5,-0.5);
\draw[mid arrow] (6.5,-0.5) -- (7,0);
\draw[mid arrow] (6.5,-0.5) -- (6,0);
\draw[mid arrow] (5.5,-1.5) --  (4.5,-0.5);
\draw[mid arrow] (4.5,-0.5) --  (4,0);
\draw[mid arrow] (4.5,-0.5) --  (5,0);
\draw[mid arrow] (3.5,-4.5) -- (3.5,-3.5);

\node at (3.5,-4.65) {$\vdots$};

\filldraw (0,0)circle(1.5pt)   (1,0)circle(1.5pt)   (2,0)circle(1.5pt)   (3,0)circle(1.5pt)   (4,0)circle(1.5pt)   (5,0)circle(1.5pt)   (6,0)circle(1.5pt)   (7,0)circle(1.5pt)   (0.5,-0.5)circle(1.5pt)   (2.5,-0.5)circle(1.5pt)   (4.5,-0.5)circle(1.5pt)   (6.5,-0.5)circle(1.5pt)   (1.5,-1.5)circle(1.5pt)   (5.5,-1.5)circle(1.5pt)   (3.5,-3.5)circle(1.5pt);
 
\node at (3.7,-3.7) {$\kappa$};
\node at (0.7,-1.7) {$\kappa_0=0\kappa$};
\node at (0,-0.7) {$00\kappa$};
\node at (-0.3,0.25) {$000\kappa$};
\node at (0.8,0.25) {$100\kappa$};
\node at (1.9,0.25) {$010\kappa$};
\node at (3,0.25) {$110\kappa$};
\node at (4,0.25) {$001\kappa$};
\node at (5.1,0.25) {$010\kappa$};
\node at (6.2,0.25) {$011\kappa$};
\node at (7.3,0.25) {$111\kappa$};
\node at (2.9,-0.7) {$10\kappa$};
\node at (5.9,-1.7) {$1\kappa$};
\node at (4.1,-0.7) {$01\kappa$};
\node at (7,-0.7) {$11\kappa$};
\end{tikzpicture}
\end{center}
\caption{Part of $\Graph_\kappa$, for some irrational $\kappa$.}
\label{fig:irrational_graph}
\end{figure}
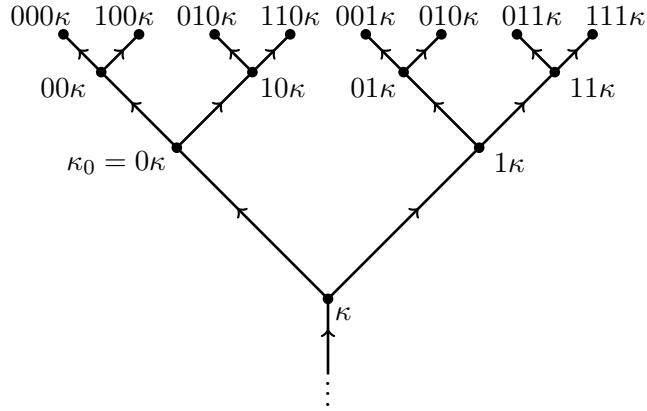

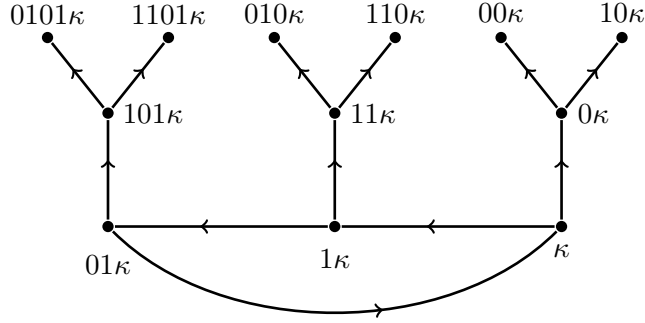
\begin{figure}[htb]
\begin{center}
\begin{tikzpicture}[scale=1, every node/.style={circle, fill=black, inner sep=1.5pt}, line width=1]

\node (a) at (0,0) [label=below:$01\kappa$] {};
\node (b) at (3,0) [label=below:$1\kappa$] {};
\node (c) at (6,0) [label=below:$\kappa$] {};

\draw[mid arrow] (b) -- (a);
\draw[mid arrow] (c) --  (b);

\node (a1) at (0,1.5) [label=right:$101\kappa$] {};
\node (b1) at (3,1.5) [label=right:$11\kappa$] {};
\node (c1) at (6,1.5) [label=right:$0\kappa$]  {};

\draw[mid arrow] (a) -- (a1);
\draw[mid arrow] (b) -- (b1);
\draw[mid arrow] (c) -- (c1);

\node (a2) at (-0.8,2.5) [label={[yshift=-11pt]above:$0101\kappa$}] {};
\node (a3) at (0.8,2.5) [label={[yshift=-11pt]above:$1101\kappa$}] {};
\draw[mid arrow] (a1) -- (a2);
\draw[mid arrow] (a1) -- (a3);

\node (b2) at (2.2,2.5) [label={[yshift=-8pt]above:$010\kappa$}] {};
\node (b3) at (3.8,2.5) [label={[yshift=-8pt]above:$110\kappa$}] {};
\draw[mid arrow] (b1) -- (b2);
\draw[mid arrow] (b1) -- (b3);

\node (c2) at (5.2,2.5)  [label={[yshift=-5pt]above:$00\kappa$}] {};
\node (c3) at (6.8,2.5) [label={[yshift=-5pt]above:$10\kappa$}] {};
\draw[mid arrow] (c1) -- (c2);
\draw[mid arrow] (c1) -- (c3);

\draw[mid arrow]
  (a) 
  .. controls (1.5,-1.5) and (4.5,-1.5) .. 
  (c);
\end{tikzpicture}

\end{center}
\caption{Part of $\Graph_\kappa$, for $\kappa$ the rational point $\overline{001}$.}
\label{fig:rational_graph}
\end{figure}

Let $G$ be a finitely generated subgroup of $V$, with $S$ a finite symmetric generating set for $G$. Fix $\kappa_0\in\Cantor$ and let $X=\kappa_0.G$, and let $\Gamma=\AG_X(G,S)$. Since $V$ and hence $G$ acts on $\Cantor$ by piecewise prefix replacements, the vertex set of $\Gamma$ is a subset of the vertex set of $\Graph\coloneqq\Graph_{\kappa_0}$. Note that edges in $\Gamma$ might not be edges in $\Graph$, but since $S$ is finite, for any edge in $\Gamma$ there is a path in $\Graph$ of uniformly bounded length connecting its endpoints. Thus the inclusion $V(\Gamma)\to V(\Graph)$ is Lipschitz.

\begin{definition}[Discrepancy]
For $\kappa,\kappa'\in V(\Graph)$, define the \emph{discrepancy} from $\kappa$ to $\kappa'$ as follows. First, if $\kappa$ and $\kappa'$ are adjacent, then define the discrepancy to be $0$ if $\kappa$ and $\kappa'$ lie in $\cycle_{\kappa_0}$ (which is only relevant if $\kappa_0$ is rational), and otherwise define it to be $1$ if the directed edge between $\kappa$ and $\kappa'$ is $(\kappa,\kappa')$ and $-1$ if it is $(\kappa',\kappa)$. Now for arbitrary $\kappa$ and $\kappa'$, choose a geodesic edge path from $\kappa$ to $\kappa'$ with vertices $\kappa=\kappa_0,\dots,\kappa_n=\kappa'$, and define the discrepancy from $\kappa$ to $\kappa'$ to be the sum of the discrepancies from $\kappa_i$ to $\kappa_{i+1}$, for $i=0,\dots,n-1$. Since geodesics are unique up to possibly having two options in $\cycle_{\kappa_0}$, this is well defined.
\end{definition}

For $\kappa\in V(\Graph)$ define
\[
f_\kappa \colon V(\Graph)\to \Z
\]
by sending $\kappa'$ to the discrepancy from $\kappa$ to $\kappa'$. Recalling that $V(\Gamma)\subseteq V(\Graph)$, the function $f_\kappa$ is naturally defined on $V(\Gamma)$ as well. We will abuse notation and also write $f_\kappa\colon \Graph\to\R$ and $f_\kappa\colon \Gamma\to\R$ for $f$ extended affinely to each edge, of either $\Graph$ or $\Gamma$. See Figure~\ref{fig:discrepancy} for an example in the irrational case.

\begin{figure}[htb]
\begin{center}
\begin{tikzpicture}[line width=1]
\draw[line width=0.75,dotted,color=gray] (-2,0) -- (8,0)   (-2,-0.5) -- (8,-0.5)   (-2,-1.5) -- (8,-1.5)   (-2,-3.5) -- (8,-3.5);
\draw[mid arrow] (3.5,-3.5) -- (1.5,-1.5);
\draw[mid arrow] (1.5,-1.5) --  (0.5,-0.5);
\draw[mid arrow] (0.5,-0.5) -- (0,0);
\draw[mid arrow] (0.5,-0.5) -- (1,0);
\draw[mid arrow] (1.5,-1.5) --  (2.5,-0.5);
\draw[mid arrow] (2.5,-0.5) --  (2,0);
\draw[mid arrow] (2.5,-0.5) --  (3,0);
\draw[mid arrow] (3.5,-3.5) -- (5.5,-1.5);
\draw[mid arrow] (5.5,-1.5) --  (6.5,-0.5);
\draw[mid arrow] (6.5,-0.5) -- (7,0);
\draw[mid arrow] (6.5,-0.5) -- (6,0);
\draw[mid arrow] (5.5,-1.5) --  (4.5,-0.5);
\draw[mid arrow] (4.5,-0.5) --  (4,0);
\draw[mid arrow] (4.5,-0.5) --  (5,0);
\draw[mid arrow] (3.5,-4.5) -- (3.5,-3.5);

\node at (3.5,-4.65) {$\vdots$};

\filldraw (0,0)circle(1.5pt)   (1,0)circle(1.5pt)   (2,0)circle(1.5pt)   (3,0)circle(1.5pt)   (4,0)circle(1.5pt)   (5,0)circle(1.5pt)   (6,0)circle(1.5pt)   (7,0)circle(1.5pt)   (0.5,-0.5)circle(1.5pt)   (2.5,-0.5)circle(1.5pt)   (4.5,-0.5)circle(1.5pt)   (6.5,-0.5)circle(1.5pt)   (1.5,-1.5)circle(2.5pt)   (5.5,-1.5)circle(1.5pt)   (3.5,-3.5)circle(1.5pt);
 
\node at (3.7,-3.7) {$\kappa$};
\node at (0.7,-1.7) {$\kappa_0=0\kappa$};
\node at (0,-0.7) {$00\kappa$};
\node at (-0.3,0.25) {$000\kappa$};
\node at (0.8,0.25) {$100\kappa$};
\node at (1.9,0.25) {$010\kappa$};
\node at (3,0.25) {$110\kappa$};
\node at (4,0.25) {$001\kappa$};
\node at (5.1,0.25) {$010\kappa$};
\node at (6.2,0.25) {$011\kappa$};
\node at (7.3,0.25) {$111\kappa$};
\node at (2.9,-0.7) {$10\kappa$};
\node at (5.9,-1.7) {$1\kappa$};
\node at (4.1,-0.7) {$01\kappa$};
\node at (7,-0.7) {$11\kappa$};

\filldraw[white] (10,0)circle(1.5pt);
\node at (-2.7,0) {$f_{\kappa_0}=2$};
\node at (-2.7,-0.5) {$f_{\kappa_0}=1$};
\node at (-2.7,-1.5) {$f_{\kappa_0}=0$};
\node at (-2.7,-3.5) {$f_{\kappa_0}=-1$};
\end{tikzpicture}
\end{center}
\caption{The graph from Figure~\ref{fig:irrational_graph}, with the $f_{\kappa_0}$ values of the vertices marked using $\kappa_0=0\kappa$.}
\label{fig:discrepancy}
\end{figure}
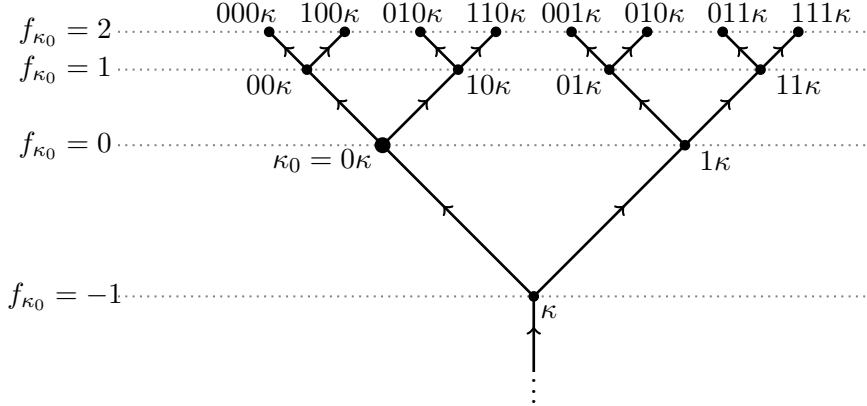

\begin{lemma}\label{lem:conn_cpts}
For each $d\in\N$ there exists a constant $A_d$ such that for all $m\in\Z$, every connected component of $([m,m+d])f_{\kappa_0}^{-1}$ in $\Gamma$ contains at most $A_d$ many vertices.
\end{lemma}

\begin{proof}
First let us prove the result for $\Graph$, writing $B_d$ for the constants. If $\kappa_0$ is irrational, then each connected component of $([m,m+d])f_{\kappa_0}^{-1}$ is a finite rooted binary tree with $2^d$ leaves, hence $2^{d+1}-1$ vertices, so we can take $B_d=2^{d+1}-1$. If $\kappa_0$ is rational then each connected component of $([m,m+d])f_{\kappa_0}^{-1}$ not meeting $\cycle_{\kappa_0}$ is a finite rooted binary tree with $2^{d+1}-1$ vertices. If a connected component of $([m,m+d])f_{\kappa_0}^{-1}$ does meet $\cycle_{\kappa_0}$ then it contains it, and consists of $\cycle_{\kappa_0}$ together with $|\cycle_{\kappa_0}|$ many rooted binary trees each of which has at most $2^{d+1}-1$ vertices. Thus in the rational case we can always take $B_d=2^{d+1}|\cycle_{\kappa_0}|$.

Now note that since $V(\Gamma)\to V(\Graph)$ is Lipschitz, if two vertices of $\Gamma$ are connected by a path in $\Gamma$ along which the $f_\kappa$ values stay between $m$ and $m+d$, then they are connected by a path in $\Graph$ along which the $f_\kappa$ values stay between $m-c$ and $m+d+c$, for some constant $c$. Thus the vertex set of a connected component of $([m,m+d])f_\kappa^{-1}$ in $\Gamma$ lies in the vertex set of a connected component of $([m-c,m+d+c])f_\kappa^{-1}$ in $\Graph$, and hence has size at most $A_d\coloneqq B_{d+2c}$. 
\end{proof}

Now we can prove Theorem~\ref{thrm:q_tree}, that $\Gamma=\AG_X(G,S)$ is quasi-isometric to a tree.

\begin{proof}[Proof of Theorem~\ref{thrm:q_tree}]
The bulk of the work will be in proving that the first homology $H_1(\Gamma)$ is \emph{uniformly generated}, meaning there exists a constant $L>0$ such that $H_1(\Gamma)$ is generated by cycles represented by closed edge paths of length at most $L$. Let us assume this for the moment, and see why the result follows. The map $V(\Gamma)\to V(\Graph)$ is injective and Lipschitz, and the graphs $\Gamma$ and $\Graph$ have bounded degree, so this map is regular in the sense of \cite[Definition~1.2]{hume22}. As explained in Subsection~1.3.3 of \cite{hume22} (see also Section~6 of \cite{benjamini12} and Proposition~2.5 of \cite{mattebon18}), this implies that the asymptotic dimension of $\Gamma$ is bounded above by that of $\Graph$, which is $1$ since $\Graph$ is an unbounded quasi-tree. By \cite{fujiwara07}, using the Manning bottleneck trick \cite{manning05}, since $H_1(\Gamma)$ is uniformly generated, $\Gamma$ is therefore quasi-isometric to a tree.

It remains to prove that for any sufficiently long closed edge path $C$ in $\Gamma$, there exist vertices $x,y\in C$ and a path in $\Gamma$ from $x$ to $y$ such that the new path is strictly shorter than either of the reduced paths from $x$ to $y$ in $C$. Fix a constant $N\in\N$ large enough that for all generators $s\in S$, the prefix replacements defining $s$ all involve prefixes of length less than $N$. Since $S$ is finite, such an $N$ exists. Note that the number of words in $\{0,1\}^*$ with length less than $N$ is $2^N-1$, so the number $2^N$ serves as a concise upper bound.

Now let $C$ be a simple closed edge path in $\Gamma$ with length (i.e., number of vertices) $|C|$. Choose vertices $p$ and $q$ in $C$ such that $p$ has minimum $f_{\kappa_0}$ value and $q$ has maximum $f_{\kappa_0}$ value, among vertices lying in $C$. Let $f$ be the function obtained from $f_{\kappa_0}$ by translating so that $(p)f=0$, and set $M_C=(q)f \in \Z$. In $C$ we now have two paths from $p$ to $q$ in $\Gamma$; designate one the ``left'' path and the other the ``right'' path, and note that the $f$ values along these paths never drop below $0$ nor go above $M_C$. Assume that $|C|$ is greater than $2 A_{2^{2N}+1}$, where $A_d$ is the constant from Lemma~\ref{lem:conn_cpts}. In particular $|C| > A_{2^{2N} + 1}$, and so Lemma~\ref{lem:conn_cpts} says that $C$ cannot be fully contained in a connected component of $([0,2^{2N} + 1])f^{-1}$, which since $C$ is connected means $M_C > 2^{2N} + 1$.

For each integer $m$ in $[1,M_C-1]$, let $\lambda_m\in V(\Gamma)$ be the latest point on the left path whose $f$ value lies in $[1,m]$. Similarly define $\rho_m$ in the right path. Note that $\lambda_m\ne \rho_m$ since $m$ is neither $0$ nor $M_C$. Consider the edge of $\Gamma$ in the left path going from the vertex $\lambda_m$ to the next vertex of the left path. In $\Graph$, these vertices are connected by a path that removes some non-empty finite prefix of length at most $N$ from $\lambda_m$ and then adds a new non-empty finite prefix of length at most $N$. Let $w_m$ be the path in $\Graph$ comprising the removal of this prefix, and similarly define $u_m$ in the right path. Since the $f$ values stay strictly above $m$ for the whole rest of the left path from $\lambda_m$ to $q$, and for the whole rest of the right path from $\rho_m$ to $q$, we know that we can obtain $\rho_m$ from $\lambda_m$ by simply replacing the prefix $w_m$ with $u_m$. To reiterate, this holds for all $m\in\Z\cap[1,M_C-1]$.

Let $\theta$ be the function from $\Z\cap[1,M_C-1]$ to the set of ordered pairs of non-empty words in $\{0,1\}^*$ of length at most $N$, sending $m$ to $(w_m,u_m)$. The domain of $\theta$ has size $M_C-1$ and the codomain of $\theta$ has size bounded above by $2^{2N}$, which is smaller than $M_C-1$, and hence $\theta$ must map some pair of distinct values to the same ordered pair. Let $1\le m_1<m_2\le M_C-1$ be arbitrary such that $(m_1)\theta=(m_2)\theta$. We know that $\lambda_{m_1}$ and $\lambda_{m_2}$ each have $w_{m_1}=w_{m_2}$ as a prefix, $\rho_{m_1}$ and $\rho_{m_2}$ each have $u_{m_1}=u_{m_2}$ as a prefix, and $\rho_{m_1}$ is obtained from $\lambda_{m_1}$ by replacing the prefix $w_{m_1}$ with $u_{m_1}$, with a similar statement in the $m_2$ case. The path in $C$ from $\lambda_{m_2}$ to $q$ via the left path and then from $q$ to $\rho_{m_2}$ through the right path in reverse encodes a product of generators from $S$ that takes $w_{m_2}\Cantor$ to $u_{m_2}\Cantor$ via prefix replacement. Since $w_{m_2}=w_{m_1}$ and $u_{m_2}=u_{m_1}$, we get a new path from $\lambda_{m_1}$ to $\rho_{m_1}$ encoded by the exact same product of generators. See Figure~\ref{fig:shortcut} for an illustration of this step.

Now we just need to argue that $m_1$ and $m_2$ can be chosen so that this path really provides a shortcut, and for this we will finally use our assumption that $|C| > 2 A_{2^{2N} + 1}$ (so far we have only used that $|C| > A_{2^{2N} + 1}$). Recall that $M_C > 2^{2N} + 1$, so $\Z\cap [M_C - 2^{2N} - 1 , M_C-1]$ is a subset of the domain of $\theta$. Since $|\Z\cap [M_C - 2^{2N} - 1 , M_C-1]| = 2^{2N} + 1$, the restriction of $\theta$ to $\Z\cap [M_C - 2^{2N} - 1 , M_C-1]$ cannot be injective. Thus we can choose $M_C - 2^{2N} - 1 \le m_1 < m_2 \le M_C-1$ such that $(m_1)\theta=(m_2)\theta$. Now the path constructed above from $x\coloneqq \lambda_{m_1}$ to $y\coloneqq \rho_{m_1}$ is shorter than the reduced path in $C$ from $x$ to $y$ passing through $q$, and we need to argue that it is also shorter than the reduced path in $C$ from $x$ to $y$ passing through $p$. The path in $C$ from $x$ to $y$ passing through $q$ lies in a connected component of $([M_C - 2^{2N} - 1,M_C])f^{-1}$, so by Lemma~\ref{lem:conn_cpts} it has at most $A_{2^{2N}+1}$ vertices. Since $|C| > 2 A_{2^{2N} + 1}$, we conclude that over half the vertices of $C$ must have $f$ value less than $m_1$. This confirms that the path in $C$ from $x$ to $y$ through $p$ is longer than the new ``shortcut'' path, since the former has length at least $|C|/2$ and the latter has the same length as the path in $C$ from $\lambda(m_2)$ to $\rho(m_2)$ through $q$, which is less than $|C|/2$.
\end{proof}

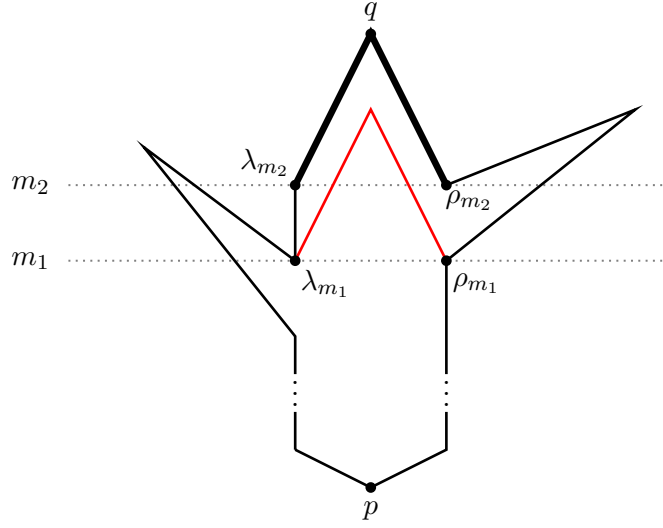
\begin{figure}[htb]
\begin{center}
\begin{tikzpicture}[line width=1]
   \draw[line width=0.75,dotted,color=gray] (-3,2.5) -- (5,2.5)   (-3,1.5) -- (5,1.5);
   \draw (0,-1) -- (0,-0.5)  (0,0) -- (0,0.5) -- (-2,3) -- (0,1.5) -- (0,2.5) -- (1,4.5) -- (2,2.5) -- (4.5,3.5) -- (2,1.5) -- (2,0)   (2,-0.5) -- (2,-1) -- (1,-1.5) -- (0,-1);
   \node at (-3.5,1.5) {$m_1$};
   \node at (-3.5,2.5) {$m_2$};
   \draw[line width=2.5] (0,2.5) -- (1,4.5) -- (2,2.5);
   \draw[color=red] (0,1.5) -- (1,3.5) -- (2,1.5);
   \filldraw (0,1.5)circle(1.5pt)   (0,2.5)circle(1.5pt)   (1,4.5)circle(1.5pt)   (2,2.5)circle(1.5pt)   (2,1.5)circle(1.5pt)   (1,-1.5)circle(1.5pt);
   \node at (1,-1.8) {$p$};
   \node at (1,4.8) {$q$};
   \node at (-0.4,2.8) {$\lambda_{m_2}$};
   \node at (0.4,1.2) {$\lambda_{m_1}$};
   \node at (2.3,2.3) {$\rho_{m_2}$};
   \node at (2.4,1.2) {$\rho_{m_1}$};
   \node at (0,-0.15) {$\vdots$};
   \node at (2,-0.15) {$\vdots$};
\end{tikzpicture}
\end{center}
\caption{An illustration of producing a ``shortcut'' in the proof of Theorem~\ref{thrm:q_tree}. The closed edge path $C$ in $\Gamma$ is black (with only the relevant vertices marked), the path in $C$ from $\lambda_{m_2}$ through $q$ to $\rho_{m_2}$ is in bold, and the new shortcut path obtained by replicating the bold path but now from $\lambda_{m_1}$ to $\rho_{m_1}$ is red.}
\label{fig:shortcut}
\end{figure}

\begin{remark}
The inclusion $V(\Gamma)\to V(\Graph)$ itself is not necessarily a quasi-isometric embedding, and so the tree to which $\Gamma$ is quasi-isometric is not necessarily a sub-quasi-tree of $\Graph$ in a natural way. To give an example, let $g,h\in V$ be the elements defined by prefix replacements as follows. First, have $g$ transpose the prefixes $00$ and $10$, and fix $01$ and $11$. Second, have $h$ replace the prefixes $000,001,01,10,110,111$ with $00,010,011,100,101,11$ respectively. Now for $\kappa=11\cdots 1100\cdots$ we have that $\kappa$ and $0\kappa$ are at distance $1$ in $\Graph$, but if the number of $1$'s in $\kappa$ is very large, then the distance from $\kappa$ to $0\kappa$ in the action graph of $\langle g,h\rangle$ is very large, since we must first do a similarly large power of $h$, then $g$, then another large power of $h$.
\end{remark}

Theorem~\ref{thrm:q_tree} quickly gives us the following result, which we will also need in the following sections.

\begin{corollary}\label{cor:no_planes}
For $\kappa_0\in\Cantor$, every subgroup of $V$ isomorphic to $\Z^2$ intersects $\Stab_V(\kappa_0)$ non-trivially.
\end{corollary}

\begin{proof}
Let $G$ be a subgroup of $V$ isomorphic to $\Z^2$. By Theorem~\ref{thrm:q_tree} the action graph of $\Z^2$ acting on $\kappa_0.G$ is quasi-isometric to a tree. Since the Cayley graph of $\Z^2$ is not quasi-isometric to a tree, $\Stab_G(\kappa_0)$ must be non-trivial.
\end{proof}

This strengthens a result that follows from work of Bleak and Salazar-D\'iaz \cite{bleak13}, that if $\Z^2\cong G\le V$ then for all non-empty open $U\subseteq \Cantor$ there exists $1\ne g\in G$ with $U.g\cap U \ne\emptyset$. Indeed, now we know that for all $\kappa\in\Cantor$ there exists $1\ne g\in G$ with $\kappa.g=\kappa$.

\begin{remark}
Finally, let us point out that there was nothing special about $V$ compared to the other Higman--Thompson groups $V_{n,r}$ ($n\ge 2$, $r\ge 1$) as in \cite{higman74}. Here $V_{n,r}$ is the group of piecewise prefix replacement homeomorphisms of the space $\Cantor_{n,r}$, which is the disjoint union of $r$ many copies of the $n$-ary Cantor set $\{0,\dots,n-1\}^\N$. All the above is the $(n,r)=(2,1)$ case, and the arguments immediately adapt to the general case. Thus we get that $\AG_X(G,S)$ is a quasi-tree, for any finitely generated $G\le V_{n,r}$ with finite symmetric generating set $S$ and $X$ a $G$-orbit in $\Cantor_{n,r}$. This can also be seen by conjugating $V_{n,r}$ to a subgroup of $V$ by an appropriate homeomorphism $\Cantor_{n,r}\to\Cantor$ and then citing Theorem~\ref{thrm:q_tree}.
\end{remark}

\section{Subgroups of $\Homeo_+(\R)$}\label{sec:linear}

Throughout this section, let $E$ denote a subgroup of the group $\Homeo_+(\R)$ of orientation-preserving homeomorphisms of $\R$. For each $a<b$ in $\R$, consider the subgroup
\[
E[a,b]\coloneqq \{f\in E\mid \Supp(f)\subseteq [a,b]\}
\]
of $E$. Note that $E[c,d]\le E[a,b]$ for all $a\le c<d\le b$. For each $a<b$, write
\[
E[a,b]' \coloneqq (E[a,b])'
\]
for the commutator subgroup of $E[a,b]$. We take the following definition from \cite{brum21}.

\begin{definition}[Locally moving]\label{def:lm}
Call $E\le \Homeo_+(\R)$ \emph{locally moving} if for all $a<b$ in $\R$ the group $E[a,b]$ has no global fixed points in $(a,b)$. Equivalently, the set of supports of elements of $E$ forms a basis for the topology on $\R$.
\end{definition}

\begin{lemma}\label{lem:comm_move}
Suppose $E$ is locally moving. Let $a<b$ in $\R$. Then the group $E[a,b]'$ has no global fixed points in $(a,b)$, and contains subgroups isomorphic to $\Z^n$ for all $n$.
\end{lemma}

\begin{proof}
Let $x\in (a,b)$. Since $E$ is locally moving we can choose $f\in E[a,b]$ with $x.f>x$, and then also choose $g\in E[a,x.f]\le E[a,b]$ with $x.g<x$. Now $x.fg=x.f>x.gf$, so $x$ is moved by $fgf^{-1}g^{-1}\in E[a,b]'$.

For the second statement, choose $a_1,b_1,a_2,b_2,\dots$ in $(a,b)$ such that $a_1<b_1<a_2<b_2<\cdots$. Since each $E[a_i,b_i]'$ acts non-trivially on $(a_i,b_i)$ by the first statement, it is non-trivial, and the $E[a_i,b_i]'$ pairwise commute. Subgroups of $\Homeo_+(\R)$ are torsion free, so the result follows.
\end{proof}

\begin{lemma}\label{lem:moving}
Suppose $E\le \Homeo_+(\R)$ is locally moving. Let $\ell<m<r$ in $\R$. Let $a<b$ and $c<d$ with $[a,b],[c,d]\subseteq(\ell,r)$ and $m\in(a,b)$. Then there exist $f_-\in E[\ell,m]'$ and $f_+\in E[m,r]'$ with $[c,d]\subseteq ([a,b])f_-f_+$.
\end{lemma}

\begin{proof}
Consider the $E[\ell,m]'$-orbit $O$ of $a$, and let $i=\inf(O)$, so $\ell\le i<m$. We claim that $i=\ell$. Suppose $i>\ell$, so by Lemma~\ref{lem:comm_move} we can choose $f_0\in E[\ell,m]'$ with $(i)f_0<i$. Since $f_0$ is a homeomorphism there exists $\varepsilon>0$ such that $(i+\varepsilon)f_0<i$. Since $i=\inf(O)$ there exists $f_1\in E[\ell,m]'$ with $(a)f_1 < i+\varepsilon$. Now $(a)f_1f_0<i$, contradicting that $i=\inf(O)$. We conclude that $i=\ell$, so in particular $i<c$, and so there exists $f_-\in E[\ell,m]'$ with $(a)f_-<c$. By a parallel argument there exists $f_+\in E[m,r]'$ with $d<(b)f_+$. Now $f_-f_+$ satisfies the desired property.
\end{proof}

\begin{definition}[Abelian germs]\label{def:ag}
Say that $E\le \Homeo_+(\R)$ \emph{has abelian germs} if for all $x\in\R$ the group of germs of $E$ at $x$ is abelian.
\end{definition}

The quintessential example of a locally moving group with abelian germs to keep in mind throughout is Thompson's group $F$. We will wait until Section~\ref{sec:examples} to get into detail on this and further examples, in particular the Stein group $F_{2,3}$.

\begin{lemma}\label{lem:LoMAG}
Suppose $E\le \Homeo_+(\R)$ is locally moving with abelian germs. Then for all $a<b$ in $\R$, every element of $E[a,b]'$ lies in some $E[c,d]'$ for $a<c<d<b$.
\end{lemma}

\begin{proof}
Let $e,f\in E[a,b]$. Since the groups of germs of $E$ at $a$ and $b$ are abelian, the support of the commutator $[e,f]$ lies in $(x,y)$ for some $a<x<y<b$. Since $E$ is locally moving we can choose $c<d$ and $h\in E$ such that $a<c<x<y<d<b$, $h$ takes $a$ to $c$ and $b$ to $d$, and $h$ fixes $(x,y)$. Now $h$ commutes with $[e,f]$ and conjugates $E[a,b]$ to $E[c,d]$, so we conclude that $[e,f]\in E[c,d]'$.
\end{proof}

\begin{proposition}\label{prop:simple}
Suppose $E\le \Homeo_+(\R)$ is locally moving with abelian germs. For any $\ell<r$ in $\R$ we have that $E[\ell,r]'$ is (non-cyclic) simple. 
\end{proposition}

\begin{proof}
Let $N$ be a non-trivial normal subgroup of $E[\ell,r]'$, and let $1\ne n\in N$. Choose $x\in(\ell,r)$ with $(x)n\ne x$, and since $n$ is a homeomorphism we can choose $a<b$ in $(\ell,r)$ such that $(a,b)\cap(a,b)n=\emptyset$. Now for any $f,g\in E[a,b]$ we have that $n^{-1}fn$ commutes with both $f$ and $g$, and so $[[f,n],g]=[f^{-1},g]$. Since $[[f,n],g]\in N$, and $f$ and $g$ were arbitrary, we conclude that $E[a,b]'\le N$ for our fixed $a<b$. Now our goal is to prove that every element of $E[\ell,r]'$ lies in $N$, which by Lemma~\ref{lem:LoMAG} is the same as proving that $E[c,d]'$ lies in $N$ for every $\ell<c<d<r$. By Lemma~\ref{lem:moving} and the fact that $E$ is locally moving, $E[c,d]'$ is contained in a conjugate of $E[a,b]'$ by an element of $E[\ell,r]'$. Hence $E[c,d]'\le N$ and we are done. That our simple group is non-cyclic is obvious from locally moving.

\end{proof}

\begin{lemma}\label{lem:gen}
Suppose $E\le \Homeo_+(\R)$ is locally moving with abelian germs. Let $x_1< x_2< \cdots <x_{n-1}< x_n$ ($n\ge 4$) be points in $\R$. Then $E[x_1,x_n]'$ is generated by its subgroups $E[x_1,x_3]',E[x_2,x_4]',\dots,E[x_{n-2},x_n]'$.
\end{lemma}

\begin{proof}
The $n>4$ case follows by repeated applications of the $n=4$ case. Thus to simplify notation, we have $\ell<a<b<r$ and we must prove that $E[\ell,r]'$ is generated by $E[\ell,b]'$ and $E[a,r]'$. Let $f\in E[\ell,r]'$, so by Lemma~\ref{lem:LoMAG} there exist $\ell<c<d<r$ with $f\in E[c,d]'$. Choose $m\in(a,b)$. By Lemma~\ref{lem:moving}, there exists $f_0$ in $E[\ell,m]' E[m,r]'$ such that $[c,d]\subseteq ([a,b])f_0$. Since $a<m<b$ we know $f_0$ lies in the subgroup generated by $E[\ell,b]'$ and $E[a,r]'$. Also, $f_0^{-1}$ conjugates $f\in E[c,d]'$ into $E[a,b]'$, which also lies in this subgroup, so we are done.
\end{proof}

\section{A semiconjugacy}\label{sec:semiconj}

Throughout this section, fix a locally moving group $E\le \Homeo_+(\R)$ with abelian germs. Also fix an embedding $\iota\colon E\to V$, and write $G$ for its image (so $G\cong E$). For each $a<b$ in $\R$ set
\[
G[a,b] \coloneqq (E[a,b])\iota \text{,}
\]
so $G[c,d]\le G[a,b]$ for all $a\le c<d\le b$. For each $[a,b]$, write $G[a,b]'\coloneqq (G[a,b])'$ for the commutator subgroup of $G[a,b]$, so $G[a,b]'=(E[a,b]')\iota$.

Our first goal is to prove a key disjointness result about the supports of $G[a,b]'$ and $G[c,d]'$ for disjoint $(a,b)$ and $(c,d)$. First we establish the following general result about regular wreath products. Given groups $A$ and $B$ (an ``acting'' group and a ``base'' group), recall that the \emph{regular wreath product} $B\wr A$ is $(\bigoplus_A B)\rtimes A$, where $A$ acts on the direct sum by permuting the coordinates via translation. A group $W$ is an \emph{internal wreath product with acting group $A$ and base $B$} if $A$ and $B$ are subgroups of $W$ that generate $W$ and the conjugates $B^a$ and $B^{a'}$ commute and intersect trivially whenever $a\ne a'$. In this case $W\cong B\wr A$.

\begin{lemma}\label{lem:disjoint_cyclic}
Let $W$ be a group that is an internal wreath product with acting group $\langle z\rangle\cong\Z$ and base $B$. Let $H\le W$ such that every non-trivial cyclic subgroup of $W$ intersects $H$ non-trivially. Then $B'\le H$.
\end{lemma}

\begin{proof}
It is sufficient to show that all commutators of elements of $B$ are elements of $H$. Let $b,b'\in B$. Choose $m>0$ such that $(zb)^m$, $(z^{-1}b')^m$, and $z^m$ all lie in $H$. Let $u\coloneqq z^{-m}(zb)^m \in H$ and note that $u$ is a product of $b$ and conjugates of $b$ by positive powers of $z$. Similarly, let $u'\coloneqq z^m(z^{-1}b')^m \in H$ and note that $u'$ is a product of $b'$ and conjugates of $b'$ by negative powers of $z$. It follows that $[b,b'] = [u,u']$, so $[b,b'] \in H$ as desired.
\end{proof}

\begin{corollary}\label{cor:disjoint_cyclic}
Let $H\le E$ such that $H$ intersects every non-trivial cyclic subgroup of $E'$ non-trivially. Then $E[a,b]'\le H$ for all $a<b$ in $\R$.
\end{corollary}

\begin{proof}
Let $f \in E[a,b]'$. Since $E$ is locally moving, by Lemma~\ref{lem:moving} we can choose $z \in E'$ such that $b<a.z$, so $(a,b)z$ is disjoint from $(a,b)$. It follows that the subgroup $W$ of $E'$ generated by $E[a,b]'$ and $z$ is an internal wreath product with acting group $\langle z\rangle$ and base group $E[a,b]'$. Now by Lemma~\ref{lem:disjoint_cyclic} applied to $E[a,b]'$ we get $E[a,b]''\le H$, and $E[a,b]'=E[a,b]''$ since $E[a,b]'$ is (non-cyclic) simple by Proposition~\ref{prop:simple}, so $f \in H$ as desired.
\end{proof}

\begin{proposition}[Key disjointness result]\label{prop:disjoint_supports}
For $a<b$ and $c<d$ in $\R$, if $(a,b)\cap(c,d)=\emptyset$ then $G[a,b]'$ and $G[c,d]'$ have disjoint supports in $\Cantor$.
\end{proposition}

\begin{proof}
Write $G_0=G[a,b]$ and $G_1=G[c,d]$ for brevity. Suppose for contradiction that $\Supp(G_0') \cap \Supp(G_1') \ne \emptyset$, say $\kappa\in \Supp(G_0') \cap \Supp(G_1')$. Set $G\coloneq\langle G_0,G_1\rangle$, so $G\cong G_0\times G_1$ since $(a,b)\cap(c,d)=\emptyset$. We will equivocate between $G$ and $G_0\times G_1$, and no confusion should arise. Let $p_0\colon G\to G_0$ and $p_1\colon G\to G_1$ be the projection maps. Our overarching goal now is to find a copy of $\Z^2$ in $G$ that intersects $\Stab_G(\kappa)$ trivially.

First suppose $(\Stab_{G_0'\times G_1}(\kappa))p_0$ equals all of $G_0'$, hence is simple by Proposition~\ref{prop:simple}. Since $\Stab_{G_0'}(\kappa)$ is normal in $\Stab_{G_0'\times G_1}(\kappa)$, we see that $\Stab_{G_0'}(\kappa)$ is normal in $(\Stab_{G_0'\times G_1}(\kappa))p_0 = G_0'$. Since we are assuming $\kappa$ is in the support of $G_0'$ we know $\Stab_{G_0'}(\kappa)$ is not all of $G_0'$, and so by simplicity $\Stab_{G_0'}(\kappa)=\{1\}$. Since $G_0'$ contains $\Z^2$ by Lemma~\ref{lem:comm_move}, we have successfully found a copy of $\Z^2$ in $G$ intersecting $\Stab_G(\kappa)$ trivially. Similarly if $(\Stab_{G_0\times G_1'}(\kappa))p_1 = G_1'$ then a parallel argument produces such a copy of $\Z^2$.

Now assume $(\Stab_{G_0'\times G_1}(\kappa))p_0 = (\Stab_G(\kappa))p_0\cap G_0'$ is properly contained in $G_0'$, which is equivalent to saying $G_0'$ is not contained in $(\Stab_G(\kappa))p_0$, and also assume that $G_1'$ is not contained in $(\Stab_G(\kappa))p_1$. By Corollary~\ref{cor:disjoint_cyclic} we can choose non-trivial $\delta_0\in G_0$ and $\delta_1\in G_1$ such that $\langle \delta_0\rangle\cap (\Stab_G(\kappa))p_0 = \{1\}$ and $\langle \delta_1\rangle\cap (\Stab_G(\kappa))p_1 = \{1\}$. In fact $\langle \delta_0,\delta_1\rangle \cap \Stab_G(\kappa) = \{1\}$ since $(\delta_0^n \delta_1^m)p_0=\delta_0^n$ and $(\delta_0^n \delta_1^m)p_1=\delta_1^m$. Thus $\langle \delta_0,\delta_1\rangle$ is a copy of $\Z^2$ in $G$ intersecting $\Stab_G(\kappa)$ trivially. In particular this is a copy of $\Z^2$ in $V$ intersecting $\Stab_V(\kappa)$ trivially, contradicting Corollary~\ref{cor:no_planes}.
\end{proof}

This result drives home why we need to work with the commutator subgroup so often rather than the whole group, since the analogous result using the whole group is not always true. For example, it is easy to find a copy of $F\times F$ inside $V$ in which the two copies of $F'$ have disjoint support but the two copies of $F$ do not, thanks to acting via a common copy of $\Z$ on some region.

\medskip

Now we will construct our $E$-semiconjugacy $\Cantor\to S^1$, where the action of $E$ on $\Cantor$ is via $\iota$ and the action of $V$, and the action of $E$ on $S^1=\R\cup\{\infty\}$ is the standard one.

\begin{definition}[The function $\phi$]\label{def:phi}
For $\kappa\in \Cantor$, if $\kappa\in\Supp(G[a,b]')$ for some $a<b$ in $\R$ then define $\phi \colon \Cantor\to S^1$ to be $(\kappa)\phi \coloneqq x$ where
\[
\{x\} = \bigcap\left\{(a,b) \mid a<b\text{ in }\R\text{ such that } \kappa \in \Supp(G[a,b]')\right\} \text{,}
\]
and otherwise define $(\kappa)\phi=\infty$.
\end{definition}

Note that $(a,b)=\Supp(E[a,b]')$ by Lemma~\ref{lem:comm_move}, so the point is that $(\kappa)\phi\in \Supp(E[a,b]')$ if and only if $\kappa\in \Supp(G[a,b]')$. Also note that $(\kappa)\phi=\infty$ if and only if $\kappa$ is not in the support of any $G[a,b]'$.

\begin{lemma}
The function $\phi$ is well defined.
\end{lemma}

\begin{proof}
Suppose $\kappa\in \Supp(G[a,b]')$ for some $a<b$, so the intersection is over a non-empty index set and $(\kappa)\phi\ne\infty$. We have to show that the intersection is non-empty and contains only one point.

If $\kappa$ lies in the supports of both $G[a,b]'$ and $G[c,d]'$, then by Proposition~\ref{prop:disjoint_supports} $(a,b)\cap(c,d)\ne\emptyset$. In particular $[a,b]\cap[c,d]\ne\emptyset$. Hence the closures of the subsets of $\R$ that we are intersecting pairwise intersect, which since they are compact means they must all have a common intersection. Now we claim that this non-empty intersection of closed intervals cannot have more than one point. Suppose $x<y$ both lie in this intersection. Choose $\ell<r$ and $z<w$ with $\ell<x<z<w<y<r$. By Lemma~\ref{lem:gen}, $E[\ell,r]'$ is generated by $E[\ell,w]'$ and $E[z,r]'$, so $G'$ is generated by $G[\ell,w]'$ and $G[z,r]'$. Thus $\kappa\in \Supp(G[\ell,r]')=\Supp(G[\ell,w]')\cup \Supp(G[z,r]')$. If $\kappa\in\Supp(G[\ell,w]')$ then $y\in [\ell,w]$, a contradiction. If $\kappa\in\Supp(G[z,r]')$ then $x\in [z,r]$, a contradiction. We conclude the intersection contains a single point.

Now we must argue that this unique point $(\kappa)\phi$ lies in the interior of all the open intervals that we are actually intersecting together. Suppose not, so without loss of generality it equals the left endpoint $a$ of some $[a,b]$ such that $\kappa\in\Supp(G[a,b]')$. Choose $g\in G[a,b]'$ with $\kappa\in\Supp(g)$. By Lemma~\ref{lem:LoMAG} $g\in G[c,d]'$ for some $a<c<d<b$. Thus $\kappa\in \Supp(G[c,d]')$, and so $a=(\kappa)\phi\in [c,d]\subseteq (a,b)$, a contradiction.
\end{proof}

Our next goal is to prove that $\phi$ is continuous, surjective, and finally a semiconjugacy. We need to impose one extra hypothesis to make this work.

\begin{definition}\label{def:FG_filtered}
Call $E$ \emph{FG-filtered} if for all $a<b$ in $\R$ the group $E[a,b]'$ is contained in a finitely generated subgroup of $E[\ell,r]'$ for some $\ell<r$ in $\R$ with $\ell<a<b<r$. Note that by Lemma~\ref{lem:moving}, as soon as this property holds for some $a<b$ it holds for all $a<b$.
\end{definition}

\begin{proposition}\label{prop:cts}
If $E$ is FG-filtered then $\phi$ is continuous.
\end{proposition}

\begin{proof}
Let $\kappa\in \Cantor$, set $x=(\kappa)\phi$, and let $U$ be an open neighborhood of $x$ in $S^1$. We must show that for all $\kappa'$ in some open neighborhood $U'$ of $\kappa$, $(\kappa')\phi$ lies in $U$. Let us first assume that $x\ne \infty$, i.e., $\kappa\in\Supp(G[\ell,r]')$ for some $\ell<r$ in $\R$, so note that $\ell<x<r$. Up to making $U$ smaller we can assume that $\infty\not\in U$. Choose $a,b,c,d$ with $\ell<a<c<x<d<b<r$ and $(a,b)\subseteq U$. We claim that $\kappa\in\Supp(G[a,b]')$. Indeed, by Lemma~\ref{lem:gen} $G[\ell,r]'$ is generated by $G[\ell,c]'$, $G[a,b]'$, and $G[d,r]'$, and $\kappa$ is fixed by the first and third of these (since $c<x<d$ and hence $x=(\kappa)\phi$ is not in $(\ell,c)$ nor $(d,r)$), so cannot also be fixed by the second. Now for any open neighborhood $U'$ of $\kappa$ contained in $\Supp(G[a,b]')$ we have $(U')\phi \subseteq (a,b) \subseteq U$ as desired.

Now suppose $x=\infty$, i.e., $\kappa$ is fixed by $G[a,b]'$ for all $a<b$. Up to making $U$ smaller we can assume it is the complement in $S^1$ of $[a,b]$ for some $a<b$ in $\R$. We want to show that some open neighborhood $U'$ of $\kappa$ is fixed by $G[a,b]'$, since then $(U')\phi$ will lie in the complement of $[a,b]$, i.e., in $U$. Since $E$ is FG-filtered we can choose $\ell<r$ in $\R$ with $\ell<a<b<r$ and finitely generated $H$ with $G[a,b]'\le H\le G[\ell,r]'$. It is well known and easy to see that the group of germs of $V$ at $\kappa$ is cyclic, and $G[\ell,r]'$ is (non-cyclic) simple by Proposition~\ref{prop:simple}, so the image of $G[\ell,r]'$, hence of $H$, in the group of germs at $\kappa$ is trivial. In particular since $H$ is finitely generated it fixes an open neighborhood $U'$ of $\kappa$, and hence $G[a,b]'$ fixes $U'$ as desired.
\end{proof}

\begin{proposition}\label{prop:surj}
If $G$ is FG-filtered then $\phi$ is surjective.
\end{proposition}

\begin{proof}
We know $\phi$ is continuous by Proposition~\ref{prop:cts}, so the image of $\phi\colon \Cantor\to S^1$ is compact, hence closed. It thus suffices to prove that the image is dense. Any non-empty open subset $U$ of $S^1$ contains an open interval $(a,b)$ for some $a<b$ in $\R$. Since $G[a,b]'$ acts non-trivially on $\Cantor$ (being a non-trivial subgroup of $V$), we can choose $\kappa\in\Supp(G[a,b]')$, and then $(\kappa)\phi\in (a,b)$ by the definition of $\phi$, so the image of $\phi$ meets $U$.
\end{proof}

The following now establishes Theorem~\ref{thrm:main}.

\begin{corollary}\label{cor:semiconj}
If $G$ is FG-filtered then $\phi \colon \Cantor\to S^1$ is an $E$-semiconjugacy.
\end{corollary}

\begin{proof}
By Propositions~\ref{prop:cts} and~\ref{prop:surj} it is continuous and surjective, so we just need to prove it is $E$-equivariant. Let $f\in E$ and set $g=(f)\iota$. We need to show that $g \circ \phi  = \phi \circ f$. Let $\kappa\in \Cantor$, write $x=(\kappa)\phi$, and write $y=(\kappa.g)\phi$, so we have to show that $x.f = y$.

First suppose $x=\infty$, so $\kappa$ lies outside the support of $\Supp(G[a,b]')$ for all $a<b$ in $\R$. Then since $\infty.f=\infty$ the goal is to prove that $y=\infty$, or equivalently that $\kappa.g$ lies outside the support of $\Supp(G[a,b]')$ for all $a<b$ in $\R$. Indeed, for any $h\in G[a,b]'$ we have $(\kappa.g).h = (\kappa.(ghg^{-1})).g = \kappa.g$ since $ghg^{-1}\in G[a.g^{-1},b.g^{-1}]'$.

Now suppose $\kappa\in\Supp(G[a,b]')$ for some $a<b$ in $\R$, so $x=(\kappa)\phi$ is the unique point lying in all $(a,b)$ for which $\kappa\in\Supp(G[a,b]')$. Since $\iota$ is a homomorphism we have $g^{-1} G[a,b]' g = (f^{-1} E[a,b]' f)\iota = (E[a.f,b.f]')\iota = G[a.f,b.f]'$ for all $a<b$, so $\kappa.g$ lies in the support of $G[a.f,b.f]'$ if and only if $\kappa$ lies in the support of $G[a,b]'$. Thus, $y=(\kappa.g)\phi$ is the unique point lying in $(a.f,b.f)$ for all $(a,b)$ for which $\kappa\in\Supp(G[a,b]')$, from which we conclude that $y=x.f$ as desired.
\end{proof}

With some additional conditions, we can prove Theorem~\ref{thrm:non_embed}.

\begin{definition}[Robust]\label{def:robust}
Call $E\le \Homeo_+(\R)$ \emph{robust} if it is locally moving with abelian germs, $E'$ is the union of the $E[a,b]'$ for $a<b$ in $\R$, and whenever $f\in E'$ has its closure of support contained in $(a,b)$ we have $f\in E[a,b]'$.
\end{definition}

\begin{definition}[Tame at $\infty$]\label{def:tame}
Call $f\in \Homeo_+(\R)$ \emph{tame at $\infty$} if there exists $a\in\R$ such that the entire ray $(a,\infty)$ is either contained in the fixed point set of $f$ or the support of $f$. Call $E\le \Homeo_+(\R)$ \emph{tame at $\infty$} if every $f\in E$ is tame at $\infty$.
\end{definition}

For example, piecewise-linear homeomorphisms (with finitely many pieces) are tame at $\infty$.

\begin{proof}[Proof of Theorem~\ref{thrm:non_embed}]
Let $\iota\colon E\to V$ be an embedding. Since $E$ is locally moving with abelian germs we can define the map $\phi\colon \Cantor\to S^1$ as in Definition~\ref{def:phi}, and since $E$ is FG-filtered, $\phi$ is a semiconjugacy by Corollary~\ref{cor:semiconj}.

Since $E$ is robust, the group of germs at $\infty$ is abelian. Thus it suffices to prove that any two non-trivial elements of the group of germs of $E$ at $\infty$ have a non-trivial common power. Let $e,f\in E$ have non-trivial germs at $\infty$, and say without loss of generality that $\infty$ is an attractor for both $e$ and $f$, that is, $x.e>x$ and $x.f>f$ for all sufficiently large $x$. By robustness, $[e^{-1},f^{-1}]$ lies in $E[a,b]'$ for some $a<b$ in $\R$. Hence $[e^{-1},f^{-1}]$ fixes $[b,\infty)$, i.e., $x.ef=x.fe$ for all $x\in[b,\infty)$. Without loss of generality $b$ is large enough that $x.e>x$ and $x.f>x$ for all $x\in[b,\infty)$. Thus for all $n,m\in\N$ and all $x\in[b,\infty)$ we have $x.e^n f^m = x.f^m e^n$, so $[e^{-n},f^{-m}]$ fixes $[b,\infty)$. Fix such an $x\in(b,\infty)$, and choose $\kappa\in (x)\phi^{-1}$. For each $n,m\in\N$, the element $[e^{-n},f^{-m}]$ has closure of support contained in $(c,d)$ for some $c<d\le b$ and hence by robustness $[e^{-n},f^{-m}]$ lies in $E[c,d]'$. Since $d\le b<x$ we have that $x=(\kappa)\phi$ lies outside $(c,d)$, so $([e^{-n},f^{-m}])\iota$ fixes $\kappa$. This holds for all $n,m$, so we conclude that $\kappa.(e^n f^m)\iota = \kappa.(f^n e^m)\iota$ for all $n,m\in\N$.

Now let $S$ be a finite symmetric generating set for $V$ containing $g=(e)\iota$ and $h=(f)\iota$, and let $\Gamma = \AG_{\kappa.V}(V,S)$. The vertex $\kappa.g^n h^m$ equals $\kappa.h^n g^m$, and hence this vertex is adjacent in $\Gamma$ to the vertices $\kappa.g^n h^{m+1}$ and $\kappa.h^m g^{n+1} = \kappa.g^{n+1} h^m$. In particular the map $(n,m)\mapsto \kappa.g^n h^m$ is a $1$-Lipschitz map from $\N^2$ to the vertex set of $\Gamma$. By the same proof as in Corollary~\ref{cor:no_planes} using $\Z^2$, this cannot be injective, and so using that $x.ef=x.fe$, $x.e>x$, and $x.f>x$, we see that there exist $n,m\in\N$ with $x.e^n=x.f^m$. Set $x_i=x.e^i$ for all $i\ge 0$, so the $x_i$ go to $\infty$. Note that
\[
x_i.f^m=x.(e^i f^m) = x.(f^m e^i) = x.(e^n e^i) = x_i.e^n \text{,}
\]
and so $f^m e^{-n}$ fixes every $x_i$. Since $E$ is tame at $\infty$, $f^m e^{-n}$ fixes some $(a,\infty)$, and so $e^n$ and $f^m$ represent the same germ at $\infty$ as desired.
\end{proof}

\section{Examples and applications}\label{sec:examples}

In this section we survey examples of subgroups of $\Homeo_+(\R)$ to which our results apply. A class of groups encompassing all our examples is the Bieri--Strebel groups, as in \cite{bieri16}, which we recall now in full generality. In practice we will primarily be interested only in certain special cases.

\begin{definition}[Bieri--Strebel groups]
Let $I\subseteq \R$ be a closed interval, $P$ a subgroup of the multiplicative group $\R_{>0}^\times$ of positive reals, and $A$ a $\Z[P]$-submodule of the additive reals $\R$. The \emph{Bieri--Strebel group} $F(I;A,P)$ (denoted $G(I;A,P)$ in \cite{bieri16}) is the group of all piecewise-linear orientation-preserving homeomorphisms of $I$ with slopes in $P$ and breakpoints in $A$.
\end{definition}

For our purposes, we will typically want to use a compact $I$ to realize $F(I;A,P)$, and then identify the interior of $I$ with $\R$ via a homeomorphism. Thus, even though our results are about groups of homeomorphisms of $\R$, one should picture the Bieri-Strebel structure as being about, say, $I=[0,1]$.

\begin{example}\label{ex:BHT_BST}
If $I=[0,1]$, $A=\Z[\frac{1}{2}]$, and $P=\langle 2\rangle$, then $F(I;A,P)$ is Thompson's group $F$. More generally if $I=[0,1]$, $A=\Z[\frac{1}{n}]$ ($n\in\N$ with $n\ge 2$), and $P=\langle n\rangle$, then $F(I;A,P)$ is the \emph{Brown--Higman--Thompson group} $F_n$ \cite{brown87,higman74}. It turns out all the $F_n$ embed in $F$; this was shown in \cite{brin98}, and also see \cite{burillo01}. Finally, if $I=[0,1]$, $P=\langle S\rangle$ for $S=\{n_1,\dots,n_s\}$ with $n_i\in\N$ and $n_i\ge 2$ for all $i$, and $A=\Z[\frac{1}{n_1\cdots n_s}]$, then $F(I;A,P)$ is the \emph{Brown--Stein--Thompson group} $F_S$ \cite{stein92}. In general the Bieri--Strebel groups can be mysterious, for example a complete picture of which ones are finitely generated remains open, but all the special cases $F$, $F_n$, and $F_S$ are finitely generated (and even type $\F_\infty$) \cite{stein92}.
\end{example}

Our main applications are that the only $F_S$ that embed in $V$ are the $F_n$, and that in this case all copies of $F_n$ in $V$ are semiconjugate to the standard one. Before approaching the proofs of these, we need to set up some important general properties of Bieri--Strebel groups.

\begin{cit}\label{cit:BS_trans}\cite[Corollary~A5.6]{bieri16}
Suppose $I\ne\R$. Let $O$ be an $F(I;A,P)$-orbit in $A\cap \mathrm{int}(I)$. Then $O$ is dense in $I$, and for all $x_1<\cdots<x_k$ and $y_1<\cdots<y_k$ in $O$, there exists $f\in F(I;A,P)$ such that $x_i.f=y_i$ for all $i$.
\end{cit}

\begin{lemma}\label{lem:coabelian}
Assume $I$ is compact and $F(I;A,P)$ is finitely generated. Then for $F(I;A,P)' \le H\le F(I;A,P)$ we have that $H$ is finitely generated if and only if it has non-trivial groups of germs at both endpoints of $I$.
\end{lemma}

\begin{proof}
Say $I=[a,b]$. Let $F(I;A,P)' \le H\le F(I;A,P)$. Consider the BNS-invariant $\Sigma^1(F(I;A,P))$ (see \cite{bieri87}). When $I=[0,1]$ this was explicitly computed in \cite{bieri87}, and in general it is easy to check that the conditions in \cite[Theorem~1.1]{strebel} hold for $F(I;A,P)$; we see that all characters $\chi\colon F(I;A,P)\to\R$ satisfying $\chi(H)=0$ lie in $\Sigma^1(F(I;A,P))$ if and only if $H$ has non-trivial groups of germs at both endpoints of $I$. Now \cite[Theorem~B1]{bieri87} says $H$ is finitely generated if and only if it has non-trivial groups of germs at both endpoints of $I$.
\end{proof}

\begin{proposition}\label{prop:BS_groups_good}
Let $F(I;A,P)$ be an arbitrary Bieri--Strebel group for $I$ compact. Let $\theta\colon \mathrm{int}(I)\to\R$ be a homeomorphism, and let $E\coloneqq (F(I;A,P))\theta \le \Homeo_+(\R)$. The following hold.
\begin{enumerate}
    \item The group $E$ is robust.
    \item Assume $E[\ell,r]$ is finitely generated for some $\ell<r$ in $(A)\theta$ with $\ell$ and $r$ in the same $E$-orbit. Then $E$ is FG-filtered.
\end{enumerate}
\end{proposition}

\begin{proof}
It is standard and easy to see that $E$ is locally moving with abelian germs. Now let $f\in E[a,b]'$ for $a<b$ in $\R$, suppose the closure of support of $f$ is in $(c,d)$ for some $a<c<d<b$, and we need to show that $f\in E[c,d]'$. Up to making $[a,b]$ larger and $[c,d]$ smaller we can assume $a,b,c,d$ all share an $E$-orbit $O$. Since the closure of support of $f$ lies in $(c,d)$, we can choose $c<x<y<d$ in $O$ with $\Supp(f)\subseteq (x,y)$. By Citation~\ref{cit:BS_trans} (filtered through $\theta$), there exists $h\in E$ taking the four points $a<x<y<b$ to the four points $c<x<y<d$, so $h$ conjugates $E[a,b]'$ to $E[c,d]'$ and fixes $f$. Now any expression writing $f$ as a product of commutators in $E[a,b]$ conjugates to an expression writing $f$ as a product of commutators in $E[c,d]$, so $f\in E[c,d]'$ as desired.

Now assume $E[\ell,r]$ is finitely generated for some $\ell<r$ in $(A)\theta$ with $\ell$ and $r$ in the same $E$-orbit. To see that $E$ is FG-filtered, it suffices to show that for $a<b$ with $\ell<a<b<r$ the group $E[a,b]'$ is contained in a finitely generated subgroup of $E[\ell,r]'$. Note that $E[a,b]'\le E[\ell,r]'\cap E[a,b]$, and we claim that $E[\ell,r]'\cap E[a,b]$ is finitely generated. By Lemma~\ref{lem:comm_move} $E[a,b]'$ has no global fixed points in $(a,b)$, so neither does $E[\ell,r]'\cap E[a,b]$, and hence $E[\ell,r]'\cap E[a,b]$ has non-trivial groups of germs at $a$ and $b$. Since $E[a,b]' \le E[\ell,r]'\cap E[a,b] \le E[a,b]$, by Lemma~\ref{lem:coabelian} (filtered through $\theta$) we conclude that $E[\ell,r]'\cap E[a,b]$ is finitely generated and we are done.
\end{proof}

Now we can prove the main results of this section.

\begin{corollary}\label{cor:F_n_semiconj}
For all $n\ge 2$, every copy of $F_n$ in $V$ is semiconjugate to the standard $F_n$ acting on $S^1$.
\end{corollary}

\begin{proof}
Note that for all $a<b$ in $[0,1]\cap \Z[\frac{1}{n}]$, $F_n[a,b]\cong F_n$ \cite[Theorem~E16.7]{bieri16}, so this is finitely generated, and hence the condition in Proposition~\ref{prop:BS_groups_good}(ii) is met. Now $F_n$ is locally moving with abelian germs and FG-filtered by Proposition~\ref{prop:BS_groups_good}, so the result follows from Theorem~\ref{thrm:main}.
\end{proof}

The next result in particular proves Corollary~\ref{cor:stein_no_V}.

\begin{corollary}\label{cor:BST_no_V}
Let $F(I;A,P)$ be a finitely generated Bieri--Strebel group satisfying the conditions of Proposition~\ref{prop:BS_groups_good}. If $P$ is non-cyclic then $F(I;A,P)$ does not embed in $V$. In particular the Stein group $F_{2,3}$ does not embed in $V$.
\end{corollary}

\begin{proof}
Fix a homeomorphism $\theta\colon \mathrm{int}(I)\to\R$ and let $E=(F(I;A,P))\theta\le\Homeo_+(\R)$. By Proposition~\ref{prop:BS_groups_good}, $E$ is robust and FG-filtered. It is also tame at $\infty$ since $F(I;A,P)$ is piecewise-linear and $I$ is compact. The group of germs of $E$ at $\infty$ is isomorphic to $P$, hence is not cyclic, so by Theorem~\ref{thrm:non_embed} $E$ does not embed in $V$. The $F_{2,3}$ result is immediate since $\langle 2,3\rangle_{\Q^\times}$ is not cyclic, and for example $F_{2,3}[1/4,3/4]\cong F_{2,3}$ is finitely generated (with the isomorphism given by conjugating by $x\mapsto \frac{x}{2} + \frac{1}{4}$).
\end{proof}

\begin{remark}
There is an interesting subgroup of $F_{2,3}$, denoted $F_{\frac{3}{2}}$, defined to be the subgroup in which every slope $2^m 3^n$ satisfies $m+n=0$, i.e., is a power of $\frac{3}{2}$. It was recently shown by Burillo and Felipe \cite{burillo24} that $F_{\frac{3}{2}}$ is finitely generated, and it remains open whether it is finitely presented. We do not know whether $F_{\frac{3}{2}}$ embeds in $V$. Another interesting Bieri--Strebel group is Cleary's irrational slope Thompson group $F_\tau=F([0,1];\Z[\tau],\langle\tau\rangle)$, where $\tau$ is the solution to $\tau^2=\tau+1$ with $\tau>1$. Since $\langle \tau\rangle$ is cyclic, we cannot rule out $F_\tau$ embedding in $V$ using our methods, but $F_\tau$ does not embed in $F$ \cite[Corollary~1]{hyde23}, and we expect it should also not embed in $V$.
\end{remark}

\medskip

Given our semiconjugacy result about all copies of $F$ in $V$, it is worth mentioning that there exist copies of $F$ in $V$ that are not conjugate in $V$ to any subgroup of the standard copy of $F$. The key is to apply a sufficiently complicated automorphism of $V$. To do this, we can use a sufficiently complicated transducer with certain nice properties. Let us illustrate this with an explicit example, shown to us by Jim Belk, Collin Bleak, and Martyn Quick.

\begin{figure}[htb]
\begin{center}
\begin{tikzpicture}[shorten >=0.5pt,node distance=3cm,on grid,auto,semithick,every state/.style={fill=blue!25!white,text=black}] 
   \node[state] (q_0)   {$q_0$}; 
   \node[state] (q_1) [below left=of q_0] {$q_1$}; 
   \node[state] (q_2) [below right=of q_0] {$q_2$}; 
    \path[->] 
    (q_0) edge [loop above] node {$0/10$} ()
          edge  [out=210,in=60]node {$1/\varnothing$} (q_1)
    (q_1) edge[bend left]  node  {$0/0$} (q_0)
          edge  [out=340,in=200]node {$1/11$} (q_2)
    (q_2) edge[bend right]  node [swap] {$0/0$} (q_0)
          edge [loop right=of q_2] node {$1/1$} ();
\end{tikzpicture}
\end{center}
\caption{The transducer representing a homeomorphism $h$ of $\Cantor$ that normalizes $V$.}
\label{fig:h-transducer}
\end{figure}
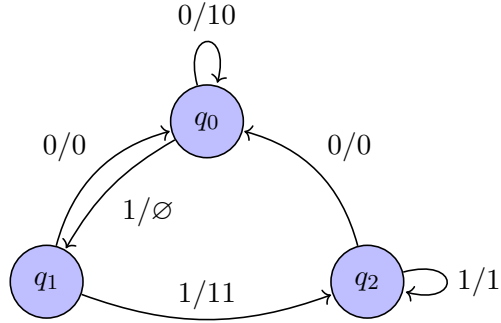

\begin{example}\label{ex:crazy_F}
Consider the transducer defined in Figure~\ref{fig:h-transducer} (with initial state $q_0$), and let $h\colon\Cantor\to\Cantor$ be the resulting homeomorphism. More precisely, for $\kappa\in\Cantor$, the diagram describes an algorithm to compute $(\kappa)h$. Say $\kappa= \epsilon_1 \epsilon_2 \epsilon_3 \cdots$ ($\epsilon_i\in \{0,1\}$). We start in state $q_0$, and feed $\kappa$ into the transducer one letter at a time. If at some stage the active state is $q$ and the next input letter is $\epsilon_i$, then we look for the transition starting at $q$ with a label of the form $\epsilon_i/\eta_i$ (where $\eta_i \in \{0,1\}^*$). We then follow that transition (updating to a new active state), while replacing $\epsilon_i$ with $\eta_i$, and get $(\kappa)h = \eta_1 \eta_2 \eta_3 \cdots$.

It was proved in \cite{bleak24} that conjugation by the homeomorphism defined by a bi-synchronizing transducer gives an automorphism of $V$, and in fact every automorphism arises as such \cite{bleak24}[Theorem~1.1]. Here \emph{bi-synchronizing} means there is some $n$ such that for both the transducer and one representing the inverse of the homeomorphism, if one inputs a word of length $n$, the resulting active state is independent of the initial state. It is straightforward to check that in our transducer, the active state after inputting a word of length $2$ is independent of the initial state, and $h$ has order $2$, so it is bi-synchronizing. Thus conjugation by $h$ is an element of $\Aut{V}$.

Suppose that $F^h$ is conjugate in $V$ to a subgroup of $F$, say $F^{hg} \le F$ for some $g \in V$. Choose $w_1,\dots,w_k$ and $u_1,\dots,u_k$ in $\{0,1\}^*$ ($k\ge 2$) such that $g$ sends $w_i\kappa$ to $u_i\kappa$ for all $i=1,\dots,k$ and $\kappa\in\Cantor$. For $z\in\{0,1\}^*$ we compute that $h$ maps $z00\Cantor$ to $z'010\Cantor$ for some $z'\in\{0,1\}^*$. Fix such $z$ and $z'$, and choose $z$ sufficiently long that $z'$ is long enough to contain one of the $w_j$ as a prefix. Let $x_0$ be the standard generator of $F$ that replaces the prefixes $0,10,11$ with $00,01,1$ respectively, and let $a$ be the element of $F$ that is the ``deferment'' of $x_0$ to $z00\Cantor$, that is, $a$ sends $z00\kappa$ to $z00((\kappa)x_0)$ and fixes all points outside $z00\Cantor$. Note that the support of $a$ equals $z00\Cantor$, so the support of $a^h$ equals $z'010\Cantor$. Now observe that $a^{hg}$ is an element of $F$, so in particular it is order preserving on its support (in the usual lexicographic order on $\Cantor$ coming from $0<1$). Since $z'$ has $w_j$ as a prefix, this implies that $a^h$ is order preserving on its support. In particular $a^h$ fixes the minimum element $z'010\overline{0}$ of $z'010\Cantor$, so $a$ fixes $(z'010\overline{0})h^{-1}$, which we compute equals $z00\overline{10}$. But $a$ does not in fact fix this, since $(z00\overline{10})a = z0001\overline{10}$, so we have reached our desired contradiction.
\end{example}

\begin{remark}
We have focused on piecewise-linear groups, and a natural question is what can be said about piecewise-projective groups, as in \cite{monod13}. For example, regarding the Lodha--Moore groups from \cite{lodha16}, one could ask whether all copies of the $F$-like Lodha--Moore group $LM_F$ inside the $V$-like one $LM_V$ are semiconjugate to the standard copy of $LM_F$. One could also define Stein-like versions of $LM_F$, and ask whether they embed in $LM_V$. We will not pursue this here, but we should note that $LM_F$ does not embed in $V$ \cite[Remark~2.9]{zaremsky16}, so the relevant questions should involve embedding piecewise-projective groups into other piecewise-projective groups, not into piecewise-linear groups.
\end{remark}

\bibliographystyle{alpha}
\newcommand{\etalchar}[1]{$^{#1}$}

\end{document}